\documentclass[a4paper]{article}

\usepackage{amsfonts}
\usepackage{amsmath}
\usepackage{mathrsfs}
\usepackage{amssymb,color}

\usepackage[latin1]{inputenc}
\usepackage{amsthm}
\usepackage{graphicx}
\usepackage{enumerate}
\newtheorem{theorem}{Theorem}[section]

\newtheorem{lemma}[theorem]{Lemma}
\newtheorem{remark}[theorem]{Remark}

\newcommand{\R}{{{\mathbb R}}}
\newcommand{\CC}{{{\mathbb C}}}

\begin{document}
\title{\textbf{An Alternative Explicit Expression of the Kernel of the One Dimensional Heat Equation with Dirichlet Conditions.\footnote{Partially supported by Xunta de Galicia (Spain), project EM2014/032 and AIE, Spain and FEDER, grant MTM2016-75140-P.}}}
\date{}
\author{Alberto Cabada\\Departamento de Estat\'{\i}stica, An\'alise Matem\'atica e Optimizaci\'{o}n,\\ Instituto de Matem\'{a}ticas, Facultade de Matem\'aticas,\\Universidade de Santiago de Compostela,\\ Santiago de Compostela, Galicia,
	Spain\\
	alberto.cabada@usc.es
} 
\maketitle

\begin{abstract}
	This paper is devoted to the study of the one dimensional non homogeneous heat equation coupled to Dirichlet Boundary Conditions.
	
	We obtain the explicit expression of the solution of the linear equation by means of a direct integral in an unbounded domain. The main novelty of this expression relies in the fact that the solution is not given as a series of infinity terms. On our expression the solution is given as a sum of two integrals with a finite number of terms on the kernel.
	
	The main novelty is that, on the contrary to the classical method, where the solutions are derived by a direct application of the separation of variables method, on the basis of the spectral theory and the Fourier Series expansion, the solution is obtained by means of the application of the Laplace Transform with respect to the time variable. As a consequence, for any $t \ge 0$ fixed, we must solve an Ordinary Differential Equation on the spatial variable, coupled to Dirichlet Boundary conditions. The solution of such a problem is given by the construction of the related Green's function.

	\vspace{.5cm}
	
	 \\
	
\end{abstract}

\section{Introduction}
\label{s-Introduction}

The aim of this paper consists on the construction of a closed expression of the solution 
\[
u: (x,t) \in [0,1] \times [0,\infty) \to u(x,t)\in \R,
\]
with $I=[0,1]$ and  $u,\, u_t, \, u_{xx} \in C(I \times [0,\infty))$, 
of the one dimensional non homogeneous heat equation coupled to Dirichlet Boundary Conditions:
 \begin{eqnarray}
 \label{e-heat}
 \frac{\partial}{\partial t} u(x,t) - \frac{\partial^2}{\partial x^2} u(x,t) & = & F(x,t), \quad t >0; \;\; x \in(0,1),\\
 \label{e-Dirichlet}
u(0,t) = u(1,t) & = & 0, \;\; \quad \qquad  t >0; \\ 
\label{e-initial}
u(x,0) & = & f(x), \qquad x\in I, 
 \end{eqnarray}

It is very well known that, under suitable regularity assumptions on the data of the equation, this problem has a unique continuous and bounded solution which, moreover, see \cite[page 210]{John}, belongs to $C^\infty( I \times (0,\infty))$. 

In particular, such uniqueness holds when $f \in C^2(I)$, $f(0)=f(1)=0$, $f''(0)=f''(1)=0$, $F(0,t)=F(1,t)=0$, $\frac{\partial^2 F}{\partial x^2}(0,t)=\frac{\partial^2 F}{\partial x^2}(1,t)=0$ for all $t \ge 0$, $F$, $\frac{\partial F}{\partial x}$, $\frac{\partial^2 F}{\partial x^2} \in C(I \times [0,\infty))$ and  bounded in $I \times [0,\infty)$.

Moreover it is given as
\[
u(x,t)=u_1(x,t)+u_2(x,t), \quad x \in I, \, t \ge 0,
\] 
with $u_1$ the unique bounded solution of problem \eqref{e-heat} -- \eqref{e-initial} with $F \equiv 0$ and $u_2$ the unique bounded solution of problem \eqref{e-heat} -- \eqref{e-initial} with $f \equiv 0$.
 
 So, the separation variables method shows us that
\begin{equation}
\label{e-u1}
u_1(x,t)= 2 \sum_{n=1}^{\infty}{ \left(\int_{0}^{1}{f(y) \sin{(n \pi y)} dy}\right) e^{-n^2 \pi^2 t} \sin{(n \pi x)}}, \; x \in I, \, t \ge 0.
\end{equation}

It is not difficult to verify that, under these assuptions,  $u_1$, $\frac{\partial u_1}{\partial t}$, $\frac{\partial u_1}{\partial x}$, $\frac{\partial^2 u_1}{\partial x^2}  \in C( I \times [0,\infty))$. 

Let's see that $u_1(x,\cdot)$ belongs to $L^1[0,\infty)$ uniformly on $I$:
\begin{equation*}
	\int_0^\infty{ \left|u_1(x,t)\right| dt} \le
2 \|f\|_\infty \sum_{n=1}^{\infty}{\int_0^\infty{ e^{-n^2 \pi^2 t}}dt} = 2 \|f\|_\infty \sum_{n=1}^{\infty}{\frac{1}{n^2 \pi^2 }} =\frac{\|f\|_\infty}{3}.
\end{equation*}
\begin{remark}
	\label{r-u1}
	Notice that, following the same arguments as before, function $ t \to e^{-a \, t}u_1(x,t)$ belongs to $L^1[0,\infty)$ uniformly on $I$ for all $a>-\pi^2$.
	\end{remark}

Moreover, due to the regularity of $\frac{\partial u_1}{\partial t} (=\frac{\partial^2 u_1}{\partial x^2})$, we have that there is a constant $K>0$ (that depends on the value of $\|f''\|_\infty$) such that
\begin{eqnarray}
\nonumber
\int_0^\infty{ \left|\frac{\partial u_1}{\partial t}(x,t)\right| dt} &=& \int_0^1{ \left|\frac{\partial u_1}{\partial t}(x,t)\right| dt}+\int_1^\infty{ \left|\frac{\partial u_1}{\partial t}(x,t)\right| dt} \\
\nonumber
&\le & \left(K + 2 \sum_{n=1}^{\infty}{\int_1^\infty{ n^2 \pi^2 e^{-n^2 \pi^2 t}}dt}\right)
\\
\nonumber
&\le &  \left(K +  \frac{2}{e^{\pi^2}-1}\right).
\end{eqnarray}

Now, by means of the Duhamel's Principle, we have that 
\begin{equation}
\label{e-u2}
u_2(x,t)= \int_{0}^{t} {u_s(x,t-s) ds},
\end{equation}
where, for each $s \ge 0$, $u_s$ is given by expression \eqref{e-u1} with  $F(y,s)$ instead of $f(y)$, that is, 
\begin{equation}
\label{e-u2-serie}
u_2(x,t)= 2 {\int_{0}^{t}\left[\sum_{n=1}^{\infty} \left(\int_{0}^{1}{F(y,s) \sin{(n \pi y)} dy}\right) e^{-n^2 \pi^2 (t-s)}  \sin{(n \pi x)}\right]ds }.
\end{equation}

However, as far as the author knows, no explicit expression in a finite number of addends is given for the Dirichlet problem \eqref{e-heat}--\eqref{e-initial}.
 
In this paper we obtain the explicit expression of the kernel that gives the unique bounded solution of the Dirichlet problem \eqref{e-heat}--\eqref{e-initial}.  The main result is obtained in Section \ref{s-expression-solution} where the complex form of such kernel is obtained by means of a direct application of the Laplace Transform. Section \ref{s-Green} is devoted to obtain the exact expression of the kernel and to study some of its qualitative properties.

\section{Closed Expression of the Solution}
\label{s-expression-solution}

In this section we obtain the exact expression of the  unique bounded solution of problem \eqref{e-heat} -- \eqref{e-initial}. Such expression is given as the integral in an unbounded domain of a finite sum of members. 

Before obtaining the explicit expression of the solution, we recall the definition of the Laplace Transform of a function $f :[0,\infty) \to \CC$:
\[
{\cal L}f(z) := \int_{0}^{\infty}{e^{-z \, t} f(t) dt}, \quad \mbox{ $z \in \CC$}.
\]

The following convergence result is very well known (see for instance \cite[Proposition 2.4.6]{almira}). 
\begin{theorem}
	\label{t-Laplace-conv}
	Let  $f :[0,\infty) \to \CC$, be a pointwise continuous function in $[0, \infty)$, for which there exist constants $K>0$ and $a \in \R$ such that
	\[
	|f(t)| \le K\, e^{a\, t}, \quad \mbox{for all $t \ge 0$}.
		\] 

Then ${\cal L}f(z)$ is well defined for all $z \in \CC$ with $Re(z) >a$.
\end{theorem}

It is immediate to verify that, if exists $f'$ on $[0,\infty)$ and its Laplace Transform is well defined, then
	\[
{\cal L}(f')(t)= t\, {\cal L}(f)(t)-f(0^+), \quad \mbox{for all $t \ge 0$}.
\]

Now, we enunciate the following version of the Inverse Transform Laplace, which is a corollary of the proof of \cite[Theorem 2.4.9]{almira}.  
\begin{theorem}
	\label{t-Laplace-inverse}
	Let $f :[0,\infty) \to \R$ be a continuous function, such that $e^{-a \,t} f(t) \in  L^1[0, \infty)$ for some $a \in \R$, and there exist the lateral derivatives of $f$ and are finite in $[0,\infty)$. Then, for any $t>0$ and $r>a$ the following identity is fulfilled:
	\begin{eqnarray*}
	f(t) &=& {\cal L}^{-1}(f)(t)  := \lim_{M \to \infty}\frac{1}{2 \, \pi} \int_{-M}^{M}{e^{(r+ i s) t} {\cal L}(f)(r+ i\, s) ds}\\
	& =: & P. V.\frac{1}{2 \, \pi} \int_{-\infty}^{\infty}{e^{(r+ i s) t} {\cal L}(f)(r+ i\, s) ds}.
	\end{eqnarray*}
\end{theorem}

Before deducing an alternative expression of function $u_1$, we introduce the following expression for $m \neq 0$
	\begin{eqnarray}
	\label{e-Green-m}
	g(m,x,y)=  \frac{-1}{m \, \sinh{m}} \,\left\{
	\begin{array}{cc}
\sinh (m (x-1)) \sinh (m y),  & 0\leq y\leq x\leq 1, \\ \\
	\sinh (m x) \sinh (m (y-1)) , & 0<x<y\leq 1, 
	\end{array}
	\right.
	\end{eqnarray}
and
\begin{equation}
\label{e-Green-0}
g(0,x,y)=\left\{
\begin{array}{cc}
(1-x) y, & 0\leq y\leq x\leq 1, \\ \\
x (1-y), & 0<x<y\leq 1. 
\end{array}
\right.
\end{equation}

It is very well know, see \cite{Cab, CCM} for details, that 
\[
v(x)=\int_{0}^{1}{  g(m,x,y) f(y) dy}, \quad x \in I,
\]
is the unique solution of the Dirichlet boundary value problem
\[
-v''(x)+m^2\, v(x)=f(x), \; x \in I, \; m \ge 0, \quad v(0)=v(1)=0.
\]

Now, we are in a position to give an alternative expression to \eqref{e-u1} of function $u_1$. The result is the following.

\begin{theorem}
	\label{t-sol-u1}
	Let $f: I  \to \R$, be such that $f \in C^2(I)$, $f(0)=f(1)=0$ and $f''(0)=f''(1)=0$. Then, the unique bounded solution $u_1:I \times [0,\infty) \to \R$ of problem \eqref{e-heat} --\eqref{e-initial}, with $F \equiv 0$, is given by the expression
	
	\begin{equation}
	\label{e-sol-w}
	u_1(x,t)=P. V.\frac{1}{2 \, \pi} \int_{-\infty}^{\infty}\int_{0}^{1}{ e^{s i t} g(\sqrt{i s},x,y) f(y) dy\, ds}.
	\end{equation}	
\end{theorem}

\begin{proof}
	First, notice that, from the assumptions on $f$, we have that $\frac{\partial u_1}{\partial t} (x,\cdot) (=\frac{\partial^2 u_1}{\partial x^2} (x,\cdot))$ is continuous in $I \times [0,\infty)$ and, since it is in $L^1[0,\infty)$ uniformly in $I$, we conclude that it is bounded in $I \times [0,\infty)$, which allows us to  define its Laplace transform with respect to the time variable. Moreover, we can also interchange the integral and the derivation with respect to the spatial variable and, as consequence, by denoting
	\[
	{\cal L}(u_1)(x,z) := \int_{0}^{\infty}{e^{-z \, t} u_1(x,t) dt}, \quad \mbox{ $z \in \CC$},
	\]
		we deduce the following property for function $u_1$:
		 	\begin{eqnarray*}
	s \, {\cal L}(u_1)(x,s) - \frac{\partial^2}{\partial x^2} {\cal L}(u_1)(x,s) & = & f(x), \quad s > 0; \;\; x \in(0,1),\\
{\cal L}(u_1)(0,s) = {\cal L}(u_1)(1,s) & = & 0, \;\;  \qquad  s > 0. 
	\end{eqnarray*}
	
	As we have mentioned previously, this problem has, for any $s > 0$ a unique solution, and it is given by 
	\[
{\cal L}(u_1)(x,s)=\int_{0}^{1}{g(\sqrt{s},x,y) f(y) dy}.
	\]
	
	Now, the result holds from  Remark \ref{r-u1} and Theorem \ref{t-Laplace-inverse} (with $r=0$).	
\end{proof}

The expression for $u_2$ follows as a direct consequence of Duhamel's principle.

\begin{theorem}
	\label{t-sol-u2}
	Let function $F: I \times [0,\infty) \to \R$ be such that
	 $F(0,t)=F(1,t)=0$,  $\frac{\partial^2 F}{\partial x^2}(0,t)=\frac{\partial^2 F}{\partial x^2}(1,t)=0$ for all $t \ge 0$, $F$, $\frac{\partial F}{\partial x}$, $\frac{\partial^2 F}{\partial x^2} \in C(I \times [0,\infty))$ and  bounded in $I \times [0,\infty)$.

	Then the unique bounded solution, $u_2:I \times [0,\infty) \to \R$, of \eqref{e-heat} -- \eqref{e-initial} is given by the expression
	\begin{equation}
	\label{e-sol-u2}
	u_2(x,t)= P. V.\frac{1}{2 \, \pi} \int_{0}^{t} \int_{-\infty}^{\infty}\int_{0}^{1}{e^{r i (t-s)} g(\sqrt{i r},x,y) F(y,s) dy\, dr \, ds}.
	\end{equation}

\end{theorem}

\section{Properties of the Kernel}

\label{s-Green}
This section is devoted to make an exhaustive study of the properties of the Green's function obtained in previous section. So, to fix ideas, we identify $\sqrt{s i} = \frac{(1+i) \sqrt{\left| s\right|} }{\sqrt{2}}$. So, we have that 

{
	\[	
g(\sqrt{i s},x,y)=  \frac{i-1}{\sqrt{2} \sqrt{\left| s\right| }\sinh\left(\frac{(1+i) \sqrt{\left| s\right| }}{\sqrt{2}}\right)} \,\left\{
\begin{array}{cc}
 \sinh
	\left(\frac{(1+i) (x-1) \sqrt{\left| s\right| }}{\sqrt{2}}\right) \sinh \left(\frac{(1+i) y
		\sqrt{\left| s\right| }}{\sqrt{2}}\right),  & 0\leq y\leq x\leq 1, \\ \\
\sinh
\left(\frac{(1+i) (y-1) \sqrt{\left| s\right| }}{\sqrt{2}}\right) \sinh \left(\frac{(1+i) x
	\sqrt{\left| s\right| }}{\sqrt{2}}\right) , & 0<x<y\leq 1. 
\end{array}
\right.
\]
}

	It is immediate to verify that the expression does not change for any other choice of $\sqrt{s i}$.

Now, one can check the following properties

\begin{lemma}
	\label{l-g-symmetric}
Function $g(\sqrt{s i},x,y)\equiv g_1(s,x,y)+ i\, g_2(s,x,y)$ defined in \eqref{e-Green-m} and \eqref{e-Green-0} satisfies the following symmetry properties:
	\begin{enumerate}
	\item $g(\sqrt{s i},x,y)=g(\sqrt{s i},y,x)$, \quad for all $s \in \R$ and $x, y \in I$. 
\item $g_1(s,x,y)=g_1(-s,x,y)$, \quad for all $s \in \R$ and $x, y \in I$. 
\item $g_2(s,x,y)=-g_2(-s,x,y)$, \quad for all $s \in \R$ and $x, y \in I$. 
	\end{enumerate}
\end{lemma}

In fact, by means of Mathematica Package, one arrives to the following expressions of the real and the imaginary part of function $g$, for any $s>0$ (for $s<0$ it will be $-s$) and $0 \le x \le y \le 1$ (in case of $0 \le y < x \le 1$ it is enough to interchange $x$ with $y$)
{\small
	\begin{eqnarray*}
 g_1(s,x,y)=  \frac{1}{\sqrt{s} \left(\cos \left(\sqrt{2} \sqrt{s}\right)-\cosh
 	\left(\sqrt{2} \sqrt{s}\right)\right)} \left( \right. \hspace{5cm}\\
\sqrt{2} \cosh \left(\frac{\sqrt{s}}{\sqrt{2}}\right) \cosh \left(\frac{\sqrt{s}
 		(x-1)}{\sqrt{2}}\right) \cosh \left(\frac{\sqrt{s} y}{\sqrt{2}}\right) \sin
 	\left(\frac{\sqrt{s}}{\sqrt{2}}\right) \sin \left(\frac{\sqrt{s} (x-1)}{\sqrt{2}}\right) \sin
 	\left(\frac{\sqrt{s} y}{\sqrt{2}}\right)\\
 -\sqrt{2} \cos \left(\frac{\sqrt{s}}{\sqrt{2}}\right)
 	\cosh \left(\frac{\sqrt{s} (x-1)}{\sqrt{2}}\right) \cosh \left(\frac{\sqrt{s} y}{\sqrt{2}}\right)
 	\sin \left(\frac{\sqrt{s} (x-1)}{\sqrt{2}}\right) \sinh \left(\frac{\sqrt{s}}{\sqrt{2}}\right) \sin
 	\left(\frac{\sqrt{s} y}{\sqrt{2}}\right)\\
 +\sqrt{2} \cos \left(\frac{\sqrt{s}
 		(x-1)}{\sqrt{2}}\right) \cosh \left(\frac{\sqrt{s}}{\sqrt{2}}\right) \cosh \left(\frac{\sqrt{s}
 		y}{\sqrt{2}}\right) \sin \left(\frac{\sqrt{s}}{\sqrt{2}}\right) \sinh \left(\frac{\sqrt{s}
 		(x-1)}{\sqrt{2}}\right) \sin \left(\frac{\sqrt{s} y}{\sqrt{2}}\right)
 \\+\sqrt{2} \cos
 	\left(\frac{\sqrt{s}}{\sqrt{2}}\right) \cos \left(\frac{\sqrt{s} (x-1)}{\sqrt{2}}\right) \cosh
 	\left(\frac{\sqrt{s} y}{\sqrt{2}}\right) \sinh \left(\frac{\sqrt{s}}{\sqrt{2}}\right) \sinh
 	\left(\frac{\sqrt{s} (x-1)}{\sqrt{2}}\right) \sin \left(\frac{\sqrt{s} y}{\sqrt{2}}\right)
 \\
 +\sqrt{2}
 	\cos \left(\frac{\sqrt{s} y}{\sqrt{2}}\right) \cosh \left(\frac{\sqrt{s}}{\sqrt{2}}\right) \cosh
 	\left(\frac{\sqrt{s} (x-1)}{\sqrt{2}}\right) \sin \left(\frac{\sqrt{s}}{\sqrt{2}}\right) \sin
 	\left(\frac{\sqrt{s} (x-1)}{\sqrt{2}}\right) \sinh \left(\frac{\sqrt{s} y}{\sqrt{2}}\right)\\
 +\sqrt{2}
 	\cos \left(\frac{\sqrt{s}}{\sqrt{2}}\right) \cos \left(\frac{\sqrt{s} y}{\sqrt{2}}\right) \cosh
 	\left(\frac{\sqrt{s} (x-1)}{\sqrt{2}}\right) \sin \left(\frac{\sqrt{s} (x-1)}{\sqrt{2}}\right) \sinh
 	\left(\frac{\sqrt{s}}{\sqrt{2}}\right) \sinh \left(\frac{\sqrt{s} y}{\sqrt{2}}\right)\\
 -\sqrt{2}
 	\cos \left(\frac{\sqrt{s} (x-1)}{\sqrt{2}}\right) \cos \left(\frac{\sqrt{s} y}{\sqrt{2}}\right) \cosh
 	\left(\frac{\sqrt{s}}{\sqrt{2}}\right) \sin \left(\frac{\sqrt{s}}{\sqrt{2}}\right) \sinh
 	\left(\frac{\sqrt{s} (x-1)}{\sqrt{2}}\right) \sinh \left(\frac{\sqrt{s} y}{\sqrt{2}}\right)\\
 +\left.\sqrt{2}
 	\cos \left(\frac{\sqrt{s}}{\sqrt{2}}\right) \cos \left(\frac{\sqrt{s} (x-1)}{\sqrt{2}}\right) \cos
 	\left(\frac{\sqrt{s} y}{\sqrt{2}}\right) \sinh \left(\frac{\sqrt{s}}{\sqrt{2}}\right) \sinh
 	\left(\frac{\sqrt{s} (x-1)}{\sqrt{2}}\right) \sinh \left(\frac{\sqrt{s} y}{\sqrt{2}}\right)\right)
\end{eqnarray*}
}
and
{\small
		\begin{eqnarray*}
	g_2(s,x,y) = \frac{1}{\sqrt{s} \left(\cos \left(\sqrt{2} \sqrt{s}\right)-\cosh
		\left(\sqrt{2} \sqrt{s}\right)\right)} \left( \right. \hspace{5cm}\\
\sqrt{2} \cosh \left(\frac{\sqrt{s}}{\sqrt{2}}\right) \cosh \left(\frac{\sqrt{s}
		(x-1)}{\sqrt{2}}\right) \cosh \left(\frac{\sqrt{s} y}{\sqrt{2}}\right) \sin
	\left(\frac{\sqrt{s}}{\sqrt{2}}\right) \sin \left(\frac{\sqrt{s} (x-1)}{\sqrt{2}}\right) \sin
	\left(\frac{\sqrt{s} y}{\sqrt{2}}\right)\\
+\sqrt{2} \cos \left(\frac{\sqrt{s}}{\sqrt{2}}\right)
	\cosh \left(\frac{\sqrt{s} (x-1)}{\sqrt{2}}\right) \cosh \left(\frac{\sqrt{s} y}{\sqrt{2}}\right)
	\sin \left(\frac{\sqrt{s} (x-1)}{\sqrt{2}}\right) \sinh \left(\frac{\sqrt{s}}{\sqrt{2}}\right) \sin
	\left(\frac{\sqrt{s} y}{\sqrt{2}}\right)\\
-\sqrt{2} \cos \left(\frac{\sqrt{s}
		(x-1)}{\sqrt{2}}\right) \cosh \left(\frac{\sqrt{s}}{\sqrt{2}}\right) \cosh \left(\frac{\sqrt{s}
		y}{\sqrt{2}}\right) \sin \left(\frac{\sqrt{s}}{\sqrt{2}}\right) \sinh \left(\frac{\sqrt{s}
		(x-1)}{\sqrt{2}}\right) \sin \left(\frac{\sqrt{s} y}{\sqrt{2}}\right)\\
+\sqrt{2} \cos
	\left(\frac{\sqrt{s}}{\sqrt{2}}\right) \cos \left(\frac{\sqrt{s} (x-1)}{\sqrt{2}}\right) \cosh
	\left(\frac{\sqrt{s} y}{\sqrt{2}}\right) \sinh \left(\frac{\sqrt{s}}{\sqrt{2}}\right) \sinh
	\left(\frac{\sqrt{s} (x-1)}{\sqrt{2}}\right) \sin \left(\frac{\sqrt{s} y}{\sqrt{2}}\right)\\
-\sqrt{2}
	\cos \left(\frac{\sqrt{s} y}{\sqrt{2}}\right) \cosh \left(\frac{\sqrt{s}}{\sqrt{2}}\right) \cosh
	\left(\frac{\sqrt{s} (x-1)}{\sqrt{2}}\right) \sin \left(\frac{\sqrt{s}}{\sqrt{2}}\right) \sin
	\left(\frac{\sqrt{s} (x-1)}{\sqrt{2}}\right) \sinh \left(\frac{\sqrt{s} y}{\sqrt{2}}\right)\\
+\sqrt{2}
	\cos \left(\frac{\sqrt{s}}{\sqrt{2}}\right) \cos \left(\frac{\sqrt{s} y}{\sqrt{2}}\right) \cosh
	\left(\frac{\sqrt{s} (x-1)}{\sqrt{2}}\right) \sin \left(\frac{\sqrt{s} (x-1)}{\sqrt{2}}\right) \sinh
	\left(\frac{\sqrt{s}}{\sqrt{2}}\right) \sinh \left(\frac{\sqrt{s} y}{\sqrt{2}}\right)\\
-\sqrt{2}
	\cos \left(\frac{\sqrt{s} (x-1)}{\sqrt{2}}\right) \cos \left(\frac{\sqrt{s} y}{\sqrt{2}}\right) \cosh
	\left(\frac{\sqrt{s}}{\sqrt{2}}\right) \sin \left(\frac{\sqrt{s}}{\sqrt{2}}\right) \sinh
	\left(\frac{\sqrt{s} (x-1)}{\sqrt{2}}\right) \sinh \left(\frac{\sqrt{s} y}{\sqrt{2}}\right)\\
-\left.\sqrt{2}
	\cos \left(\frac{\sqrt{s}}{\sqrt{2}}\right) \cos \left(\frac{\sqrt{s} (x-1)}{\sqrt{2}}\right) \cos
	\left(\frac{\sqrt{s} y}{\sqrt{2}}\right) \sinh \left(\frac{\sqrt{s}}{\sqrt{2}}\right) \sinh
	\left(\frac{\sqrt{s} (x-1)}{\sqrt{2}}\right) \sinh \left(\frac{\sqrt{s} y}{\sqrt{2}}\right)\right).
\end{eqnarray*}
}

Despite the expression of the real and imaginary parts of the Green's function is difficult to deal with, the value of the Modulus of the complex number is relatively easy to manage. Using  Mathematica package again, one can verify that, for any $s >0$, it is given, for $0 \le y\le x \le 1$, (for $s<0$ it will be $-s$ and for $0 \le y < x \le 1$ it is enough to interchange $x$ with $y$) by the expression 
{
	\begin{eqnarray*}
|g(\sqrt{s i},x,y)|= 
 \sqrt{\frac{\left(\cos \left(\sqrt{2} \sqrt{s} (x-1)\right)-\cosh \left(\sqrt{2} \sqrt{s}
			(x-1)\right)\right) \left(\cos \left(\sqrt{2} \sqrt{s} y\right)-\cosh \left(\sqrt{2} \sqrt{s}
			y\right)\right)}{2\, s \left(\cosh \left(\sqrt{2} \sqrt{s}\right)-\cos \left(\sqrt{2}
			\sqrt{s}\right)\right)}}	
		\end{eqnarray*}
}

Moreover $|g(0,x,y)|= y(1-x)$ for all $x,y \in I$.

As a consequence, for any $s\in\R$ and $x \in I$, it is obvious that
\[
\max_{y \in I}{\{|g(\sqrt{s i},x,y)|\}} = |g(\sqrt{s i},x,x)|
\]
and that
\[	
\lim_{|s|\to \infty} \sqrt{|s|} \,|g(\sqrt{s i},x,x)|= \frac{1}{2} \quad \mbox{for all $x \in (0,1)$}.
\]

Moreover, due to properties of symmetry showed in Lemma \ref{l-g-symmetric}, it is obtained the following expression for $u_1$ and $u_2$:

\[
u_1(x,t)=\frac{1}{\pi} \int_0^\infty \int_0^1 \left(g_1(s,x,y) \cos{s\,t} - g_2(s,x,y)\sin{s\,t}  \right) f(y) dy ds
\]
and
\[
u_2(x,t)=\frac{1}{\pi} \int_0^t \int_0^\infty \int_0^1 \left(g_1(r,x,y) \cos{r\,(t-s)} - g_2(r,x,y)\sin{r\,(t-s)}  \right) F(y,s) dy dr ds
\]

\subsection{Final Remarks}
We point out that the expression obtained for both $u_1$ and $U_2$ holds for funcions $f$ and $F$ under weaker assuptions as the one impossed in Theorems \ref{t-sol-u1} and \ref{t-sol-u2}. They make sense if such integrals, and the correspondent derivatives with respect to  $t$, are well defined.

In particular, it is not needed either the boundedness or the continuity at $t=0$ of function $F$.

Moreover, we can also consider the discontinuous and unbounded function $F(x,t)=\frac{\left(8 \pi ^2 t+1\right) \sin (2 \pi  x)}{2 \sqrt{t}}$, for which equation \eqref{e-sol-u2} gives us the unbouded solution
\[u_2(x,t)=\sqrt{t} \sin (2 \pi  x).\]

\end{document}